\documentclass[11pt, notitlepage]{article}   

\usepackage{amsmath,amsthm,amsfonts}   

\usepackage{amscd}
\usepackage{amsfonts}
\usepackage{amssymb}
\usepackage{color}      
\usepackage{epsfig}
\usepackage{graphicx}           

\usepackage{tikz}

\theoremstyle{plain}
\newtheorem{theorem}{Theorem}[section]
\newtheorem{lemma}[theorem]{Lemma}
\newtheorem{corollary}[theorem]{Corollary}
\newtheorem{proposition}[theorem]{Proposition}
\newtheorem{observation}[theorem]{Observation}
\newtheorem{question}[theorem]{Question}

\newtheorem{porism}[theorem]{Porism}

\theoremstyle{definition}
\newtheorem{definition}[theorem]{Definition}

\newcommand{\code}{\textnormal{code}}
\newcommand{\cp}{\,\square\,}

\def\finf{\mathop{{\rm I}\kern -.27 em {\rm F}}\nolimits}

\begin{document}


\title{Maker-Breaker resolving game}

\author{
Cong X. Kang$^1$ \and Sandi Klav\v{z}ar$^{2,3,4}$ \and Ismael G. Yero$^5$ \and Eunjeong Yi$^1$\medskip\\
\small Texas A\&M University at Galveston, Galveston, TX 77553, USA$^{1}$\\
\small Faculty of Mathematics and Physics, University of Ljubljana, Slovenia$^2$\\
\small Faculty of Natural Sciences and Mathematics, University of Maribor, Slovenia$^3$\\
\small Institute of Mathematics, Physics and Mechanics, Ljubljana, Slovenia$^4$\\
\small Universidad de C\'{a}diz, Av. Ram\'{o}n Puyol s/n, 11202 Algeciras, Spain$^5$\\
{\small\tt kangc@tamug.edu}; {\small\tt sandi.klavzar@fmf.uni-lj.si} \\
{\small\tt ismael.gonzalez@uca.es}; {\small\tt yie@tamug.edu}}

\maketitle

\date{}

\begin{abstract}
A set of vertices $W$ of a graph $G$ is a resolving set if every vertex of $G$ is uniquely determined by its vector of distances to $W$. In this paper, the Maker-Breaker resolving game is introduced. The game is played on a graph $G$ by Resolver and Spoiler who alternately select a vertex of $G$ not yet chosen. Resolver wins if at some point the vertices chosen by him form a resolving set of $G$, whereas Spoiler wins if the Resolver cannot form a resolving set of $G$. The outcome of the game is denoted by $o(G)$ and  $R_{\rm MB}(G)$ (resp.\ $S_{\rm MB}(G)$) denotes the minimum number of moves of Resolver (resp.\ Spoiler) to win when Resolver has the first move. The corresponding invariants for the game when Spoiler has the first move are denoted by $R'_{\rm MB}(G)$ and $S'_{\rm MB}(G)$. Invariants $R_{\rm MB}(G)$, $R'_{\rm MB}(G)$, $S_{\rm MB}(G)$, and $S'_{\rm MB}(G)$ are compared among themselves and with the metric dimension ${\rm dim}(G)$.  A large class of graphs $G$ is constructed for which $R_{\rm MB}(G) > {\rm dim}(G)$ holds. The effect of twin equivalence classes and pairing resolving sets on the Maker-Breaker resolving game is described. As an application $o(G)$, as well as $R_{\rm MB}(G)$ and $R'_{\rm MB}(G)$ (or $S_{\rm MB}(G)$ and $S'_{\rm MB}(G)$), are determined for several graph classes, including trees, complete multi-partite graphs, grid graphs, and torus grid graphs.
\end{abstract}

\noindent\small {\bf{Keywords:}} resolving set; metric dimension; Maker-Breaker game; Maker-Breaker resolving game; twin equivalence class, pairing resolving set\\
\small {\bf{2020 Mathematics Subject Classification:}} 05C12, 05C57, 05C69

\section{Introduction}

Let $G = (V(G), E(G))$ be a finite, simple, undirected, connected graph  of order at least $2$. A set $W \subseteq V(G)$ is a \emph{resolving set} of $G$ if, for every pair of distinct vertices $x$ and $y$ of $G$, there exists $z \in W$ such that $d(x,z) \neq d(y,z)$, where $d(u,v)$ denotes the shortest-path distance between $u$ and $v$. The \emph{metric dimension} $\dim(G)$ of $G$ is the minimum of the cardinalities over all resolving sets of $G$. A resolving set of cardinality $\dim(G)$ is called a \emph{metric basis} for $G$. These concepts were independently introduced by Slater~\cite{slater} and by Harary and Melter~\cite{harary}. Soon after it was noted in~\cite{NP} that determining the metric dimension of a graph is an NP-hard problem. Metric dimension has found applications in fields as diverse as robot navigation, network discovery and verification, chemistry, combinatorial optimization, and strategies for the mastermind game. See~\cite{bailey-2011, chartrand-2003} for history and surveys and~\cite{akhter-2019, jiang-2019, toit-2019} for some of the more recent results on metric dimension.

The Maker-Breaker game, introduced in 1973 by Erd\H{o}s and Selfridge~\cite{erdos-1973}, is played on an arbitrary hypergraph $H = (V,E)$. Two players, named Maker and Breaker, alternately select a vertex from $V$ not yet chosen in the course of the game. Maker wins the game if he is able to select all the vertices of one of the hyperedges from $E$, while Breaker wins if she is able to prevent Maker from doing so. We refer to the books of Beck~\cite{beck-2008} and of Hefetz et al.~\cite{hefetz-2014} for more information on this game as well as to papers~\cite{hancock-2019, mikalacki-2018} for recent related developments.

Motivated by the Maker-Breaker game and the domination game~\cite{bresar-2010}, Duch\^ene, Gledel, Parreau, and Renault introduced the Maker-Breaker domination game~\cite{duchene-2020}. This game is played on a graph $G$ and can be described as the Maker-Breaker game on the hypergraph with the same vertex set as $G$ and with hyperedges corresponding to the dominating sets of $G$. The game was further investigated in~\cite{gledel-2019}, while in~\cite{gledel-2019+} its total version was  introduced. Inspired by these developments, we introduce in this paper the \emph{Maker-Breaker resolving game} ({\em MBRG} for short) as follows.

The MBRG is played on a graph $G$ by two players, Resolver and Spoiler, which will be denoted throughout the paper by ${\rm R}^*$ and ${\rm S}^*$, respectively. ${\rm R}^*$ and ${\rm S}^*$ alternately select (without missing their turn) a vertex of $G$ that was not yet chosen in the course of the game. If ${\rm R}^*$ is the first to play, we speak of an {\em R-game}, otherwise we have an {\em S-game}. ${\rm R}^*$ wins if at some point the vertices ${\rm R}^*$ has chosen form a resolving set of $G$, whereas ${\rm S}^*$ wins if ${\rm R}^*$  cannot form a resolving set of $G$. The \emph{outcome} of the MBRG on a graph $G$ is denoted by $o(G)$, and there are four possible outcomes as follows: (1) $o(G)=\mathcal{R}$, if ${\rm R}^*$ has a winning strategy in the R-game and the S-game; (2) $o(G)=\mathcal{S}$, if ${\rm S}^*$ has a winning strategy in the R-game and the S-game; (3) $o(G)=\mathcal{N}$, if the first player has a winning strategy; (4) $o(G)=\widetilde{\mathcal{N}}$, if the second player has a winning strategy.

Now, suppose a company $X$ tries to secure its network by installing transmitters at certain locations within the company, so that the robot is aware of its security status at all times and thus identifying the exact location (or a specific computer with  virus infection) in the network, whereas a rival company $Y$ tries to prevent $X$ from forming a secure network by occupying or controlling strategic locations or computers within the network of $X$. With this application in mind and considering the time constraint (the longer it takes for a player to win a game, the more it costs for the player), we introduce the following terminology and notation.
\begin{itemize}
\item The \emph{Maker-Breaker resolving number} $R_{\rm MB}(G)$ of $G$ is the minimum number of moves of ${\rm R}^*$ to win the R-game provided he has a winning strategy. Otherwise, we set $R_{\rm MB}(G) = \infty$.
\item $R'_{\rm MB}(G)$ is the minimum number of moves of ${\rm R}^*$ to win the S-game provided he has a winning strategy. Otherwise, we set $R'_{\rm MB}(G) = \infty$.
\item The \emph{Maker-Breaker spoiling number} $S_{\rm MB}(G)$ of $G$ is the minimum number of moves of ${\rm S}^*$ to win the R-game provided she has a winning strategy. Otherwise, we set $S_{\rm MB}(G) = \infty$.
\item $S'_{\rm MB}(G)$ is the minimum number of moves of ${\rm S}^*$ to win the S-game provided she has a winning strategy. Otherwise, we set $S'_{\rm MB}(G) = \infty$.
\end{itemize}

This paper is organized as follows. In the next section we obtain some general results on the outcome of the MBRG. In  Section~\ref{sec:pairing} the effect of twin equivalence classes and pairing resolving sets on the MBRG is described and as an application a large class of graphs $G$ is constructed for which $R_{\rm MB}(G) > \dim(G)$ holds. In Section~\ref{sec:applications} we determine $o(G)$, as well as $R_{\rm MB}(G)$ and $R'_{\rm MB}(G)$ or $S_{\rm MB}(G)$ and $S'_{\rm MB}(G)$, when $G$ is a tree, the Petersen graph, a bouquet of cycles, a complete multi-partite graph, a grid graph, or a torus grid graph.

\section{Some general properties of the MBRG}
\label{sec:general}

In this section we compare parameters of the MBRG with the metric dimension, $R_{\rm MB}(G)$ with $R'_{\rm MB}(G)$, and $S_{\rm MB}(G)$ with $S'_{\rm MB}(G)$. Along the way we prove the so-called No-Skip Lemma for the Maker-Breaker game played on a hypergraph. But first we comment on the possible outcomes of the MBRG.

Among the four possible outcomes listed in the introduction, the outcome $o(G)=\widetilde{\mathcal{N}}$ never occurs as follows from a general result on the Maker-Breaker game, cf.~\cite{beck-2008, hefetz-2014}. In the case of the Maker-Breaker domination game, this statement and its proof are given in~\cite[Proposition 2]{duchene-2020}. The same argument applies also to the Maker-Breaker resolving game. The other three possible outcomes in the latter game are realized, as the reader can verify on the examples given in Fig.~\ref{fig_example}.

\begin{figure}[ht!]
\centering
\begin{tikzpicture}[scale=.8, transform shape]

\node [draw, shape=circle, scale=.8] (1) at  (0, 0) {};
\node [draw, shape=circle, scale=.8] (2) at  (1, 0) {};
\node [draw, shape=circle, scale=.8] (3) at  (2, 0) {};
\node [draw, shape=circle, scale=.8] (4) at  (3, 0) {};

\node [draw, shape=circle, scale=.8] (a) at  (8, 0) {};
\node [draw, shape=circle, scale=.8] (a1) at  (6.5, -1) {};
\node [draw, shape=circle, scale=.8] (a2) at  (7.5, -1) {};
\node [draw, shape=circle, scale=.8] (a3) at  (8.5, -1) {};
\node [draw, shape=circle, scale=.8] (a4) at  (9.5, -1) {};

\node [draw, shape=circle, scale=.8] (b) at  (14, 0) {};
\node [draw, shape=circle, scale=.8] (b1) at  (12.5, -1) {};
\node [draw, shape=circle, scale=.8] (b2) at  (13.5, -1) {};
\node [draw, shape=circle, scale=.8] (b3) at  (14.5, -1) {};
\node [draw, shape=circle, scale=.8] (b4) at  (15.5, -1) {};
\node [draw, shape=circle, scale=.8] (b5) at  (15.5, -2) {};

\node [scale=1] at (1.5,0.7) {$G_1$};
\node [scale=1] at (8,0.7) {$G_2$};
\node [scale=1] at (14,0.7) {$G_3$};

\node [scale=1] at (1.5,-2) {$o(G_1)=\mathcal{R}$};
\node [scale=1] at (8,-2) {$o(G_2)=\mathcal{S}$};
\node [scale=1] at (14,-2) {$o(G_3)=\mathcal{N}$};

\draw(1)--(2)--(3)--(4);\draw(a1)--(a)--(a4);\draw(a2)--(a)--(a3);\draw(b1)--(b)--(b4)--(b5);\draw(b2)--(b)--(b3);

\end{tikzpicture}
\caption{Three examples realizing the three outcomes.}\label{fig_example}
\end{figure}
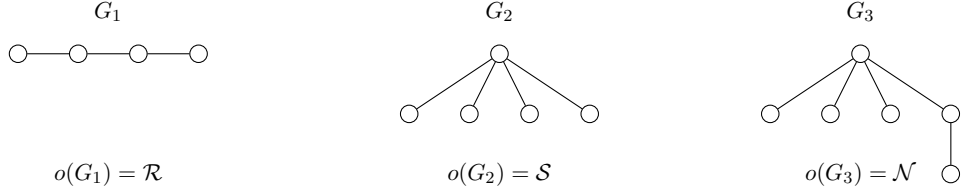

For further examples, note that if $n \ge 2$, then  $o(P_n)=\mathcal{R}$ and $R_{\rm MB}(P_n)=R'_{\rm MB}(P_n)=\dim(P_n)=1$. Moreover, $o(C_3) = \mathcal{N}$, and if $n\ge 4$, then $o(C_n) = \mathcal{R}$ and $R_{\rm MB}(C_n)=R'_{\rm MB}(C_n)=\dim(C_n)=2$.

The order of a graph $G$ will be denoted by $n(G)$. We have the following simple relations between the outcome of the MBRG and the metric dimension.

\begin{proposition}\label{propR}
If $G$ is a connected graph, then the following properties hold.
\begin{itemize}
\item[\emph{(i)}] If $o(G)=\mathcal{R}$, then $\dim(G)\leq \lfloor\frac{n(G)}{2}\rfloor$.
\item[\emph{(ii)}] If $\dim(G) \ge \lceil\frac{n(G)}{2}\rceil+1$, then $o(G)=\mathcal{S}$.
\end{itemize}
\end{proposition}

\begin{proof}
(i) Suppose that $o(G)=\mathcal{R}$ and consider the S-game. After the game is finished, ${\rm R}^*$ has clearly selected at most $\lfloor\frac{n(G)}{2}\rfloor$ vertices. As the set of vertices selected by ${\rm R}^*$ forms a resolving set of $G$, we conclude that $\dim(G)\leq \lfloor\frac{n(G)}{2}\rfloor$.

(ii) No matter whether the R-game or the S-game is played, ${\rm R}^*$ selects at most $\lceil\frac{n(G)}{2}\rceil$ vertices by the end of the game. As $\dim(G) \ge \lceil\frac{n(G)}{2}\rceil+1$, these vertices do not form a resolving set of $G$, hence ${\rm S}^*$ wins the R-game as well as the S-game.
\end{proof}

If $n\ge4$, then $\dim(K_n)=n-1 \ge \lceil\frac{n}{2}\rceil+1$, thus $o(K_n)=\mathcal{S}$ by Proposition~\ref{propR}(ii). Similarly, let $B_m$ be the graph obtained from $m\ge 5$ disjoint copies of $C_4$ by identifying a vertex from each $C_4$ at a common vertex. Then $n(B_m) = 3m+1$ and $\dim(B_m)=2m-1 \ge \left\lceil\frac{n(B_m)}{2}\right\rceil+1$; thus $o(B_m)=\mathcal{S}$ by Proposition~\ref{propR}(ii).

Next, we compare $R_{\rm MB}(G)$ with $R'_{\rm MB}(G)$ and $S_{\rm MB}(G)$ with $S'_{\rm MB}(G)$. To this end, we  consider the possibility that a player is allowed to skip a move; equivalently, a player allows the other player to select two vertices in one move. The observation that skipping offers no advantage to a player in the Maker-Breaker domination game was proved in~\cite{gledel-2019}. We next show that a parallel argument works for the Maker-Breaker game played on an arbitrary hypergraph.

\begin{lemma} {\rm (No-Skip Lemma)}
\label{lem:no-skip}
If the Maker-Breaker game is played on a hypergraph $H$, then in an optimal strategy of ${\rm R}^*$ to win in the minimum number of moves it is never an advantage for him to skip a move. Moreover, it never disadvantages ${\rm R}^*$ for ${\rm S}^*$ to skip a move.
\end{lemma}

\begin{proof}
Suppose the R-game or the S-game is played. Let ${\rm R}^*$ and ${\rm S}^*$ play optimally until ${\rm S}^*$ decides to skip a move. Then ${\rm R}^*$ imagines that ${\rm S}^*$ played an arbitrary legal move $w$, and replies optimally. ${\rm R}^*$ continues to use this strategy until the end of the game. It may happen that in the course of the game ${\rm S}^*$ selects a vertex which was already selected in the imagined game of ${\rm R}^*$. In that case, ${\rm R}^*$ imagines that  some other legal move has been played by ${\rm S}^*$. In this way, the game on $G$ will finish in no more than the minimum number of moves played in the usual Maker-Breaker game. With a strategy of ${\rm S}^*$ parallel to the above strategy of ${\rm R}^*$, it also follows that it is never an advantage for ${\rm R}^*$ to skip a move.
\end{proof}

No-Skip Lemma quickly implies the announced comparison of $R_{\rm MB}(G)$ with $R'_{\rm MB}(G)$ and $S_{\rm MB}(G)$ with $S'_{\rm MB}(G)$.

\begin{proposition}\label{prop:relations}
If $G$ is a connected graph, then the following properties hold.
\begin{itemize}
\item[\emph{(i)}] If $o(G)=\mathcal{R}$, then $R'_{\rm MB}(G) \ge R_{\rm MB}(G) \ge \dim(G)$.
\item[\emph{(ii)}] If $o(G)=\mathcal{S}$, then $S_{\rm MB}(G) \ge S'_{\rm MB}(G)$.
\end{itemize}
\end{proposition}

\begin{proof}
(i) The R-game can be viewed as the S-game in which ${\rm S}^*$ has skipped her first move. Hence the first inequality follows from Lemma~\ref{lem:no-skip} specialized to the MBRG. The second inequality follows from the fact that when the MBRG is finished, the set of vertices selected by ${\rm R}^*$ forms a resolving set of $G$.

(ii) The S-game can be viewed as the R-game in which ${\rm R}^*$ has skipped his first move, and thus the inequality follows.
\end{proof}

For the lexicographic product graph $G=C_m[K_2]$, where $m \ge 4$, we have $\dim(G)=R_{\rm MB}(G)=R'_{\rm BM}(G)=m=\left\lfloor\frac{n(G)}{2}\right\rfloor$ (see~\cite{hammack-2011} for the definition of the lexicographic product). This shows the sharpness of the bounds of Proposition~\ref{prop:relations}(i). On the other hand, at the end of Section~\ref{sec:pairing} we will construct a large family of graphs $G$ for which $R_{\rm MB}(G) >  \dim(G)$ holds.

\section{Twin equivalence classes and pairing resolving sets}
\label{sec:pairing}

In this section we consider two concepts that are very useful when dealing with the MBRG; this fact will be demonstrated in the rest of the paper.

The \emph{open neighborhood} of a vertex $v \in V(G)$ is $N(v)=\{u \in V(G) \mid uv \in E(G)\}$. Vertices $u$ and $v$ are \emph{twins} if $N(u)\setminus \{v\}=N(v)\setminus \{u\}$; notice that a vertex is its own twin. Hernando et al.~\cite[Lemma 2.7]{Hernando} observed that the twin relation is an equivalence relation and that an equivalence class under it, hereafter called \emph{a twin equivalence class}, induces either a clique or an independent set. We recall the following well-known fact.

\begin{observation} {\emph {\cite[Corollary 2.4]{Hernando}}}\label{obs_twin}
If $W$ is a resolving set of $G$ and $u$ and $v$ are distinct members of the same twin equivalence class of $G$,  then $W \cap \{u, v\} \neq \emptyset$.
\end{observation}

Here is now a relation of twin equivalence classes with the MBRG.

\begin{proposition}\label{prop_twin}
Let $G$ be a connected graph with $n(G)\ge 4$.
\begin{itemize}
\item[{\em (a)}] If $G$ has a twin equivalence class of cardinality at least $4$, then $o(G)=\mathcal{S}$ and $S_{\rm MB}(G)=S'_{\rm MB}(G)=2$.
\item[{\em (b)}] If $G$ has two distinct twin equivalence classes of cardinality at least $3$, then $o(G)=\mathcal{S}$ and $S_{\rm MB}(G)=S'_{\rm MB}(G)=2$.
\end{itemize}
\end{proposition}

\begin{proof}
Let $W$ be a resolving set of $G$.

(a) Let $Q = \{u_1,\ldots, u_k\} \subseteq V(G)$ be a twin equivalence class of $G$, where $k\ge 4$. Then $|W \cap Q| \ge k-1$ by Observation~\ref{obs_twin}. Since $k \ge 4$, we infer that ${\rm S}^*$ can occupy two vertices of $Q$ after her second move, regardless of whether ${\rm S}^*$ plays first or second. So, ${\rm R}^*$ can occupy at most $k-2$ vertices of $Q$, and thus ${\rm R}^*$ fails to occupy vertices that form a resolving set of $G$. Thus $o(G)=\mathcal{S}$ and $S_{\rm MB}(G)=S'_{\rm MB}(G)=2$.

(b) Let $Q_1'$ and $Q_2'$ be different twin equivalence classes of $G$, each of cardinality at least $3$. Clearly, $Q_1' \cap Q_2' = \emptyset$. Let $Q_1=\{u_1, u_2, u_3\}\subseteq Q_1'$ and $Q_2=\{u'_1, u'_2, u'_3\}\subseteq Q_2'$. Then $|W \cap Q_1| \ge 2$ and $|W \cap Q_2| \ge 2$ by Observation~\ref{obs_twin}. Note that ${\rm S}^*$ can occupy three vertices of $Q_1 \cup Q_2$ after her third move, regardless of whether ${\rm S}^*$ plays first or second. So, after the third move by ${\rm S}^*$, there are the following four possibilities: (i) ${\rm S}^*$ occupies all vertices of $Q_1$; (ii) ${\rm S}^*$ occupies two vertices of $Q_1$ and one vertex of $Q_2$; (iii) ${\rm S}^*$ occupies one vertex of $Q_1$ and two vertices of $Q_2$; (iv) ${\rm S}^*$ occupies all three vertices of $Q_2$. In each case, ${\rm R}^*$ fails to occupy vertices that form a resolving set of $G$. Thus $o(G)=\mathcal{S}$.

Next, we determine $S_{\rm MB}(G)$ and $S'_{\rm MB}(G)$. If ${\rm S}^*$ plays first, then after her second move, ${\rm S}^*$ can occupy two vertices of $Q_1$. If ${\rm S}^*$ plays second, then after her second move, ${\rm S}^*$ can occupy two vertices of $Q_1$ (if ${\rm R}^*$ occupies a vertex in $Q_2$ in his first move) or ${\rm S}^*$ can occupy two vertices of $Q_2$ (if ${\rm R}^*$ occupies a vertex in $Q_1$ in his first move). So, $S_{\rm MB}(G)=S'_{\rm MB}(G)=2$.
\end{proof}

Let $[k]$ denote the set $\{1,\ldots, k\}$. Let $A=\{\{u_1,w_1\}, \ldots, \{u_k, w_k\}\}$ be a set of $2$-subsets of $V(G)$ such that $|\cup_{i=1}^k\{u_i, w_i\}|=2k$. We say that $A$ is a \emph{pairing resolving set} of $G$ if every set $\{x_1,\ldots, x_k\}$, where $x_i\in \{u_i, w_i\}$ and $i\in [k]$, is a resolving set of $G$.

\begin{proposition}\label{prop:pairing-in-hypergraphs}
If a graph $G$ admits a pairing resolving set, then $o(G)=\mathcal{R}$.
\end{proposition}

\begin{proof}
Let $A$ be a pairing resolving set of $G$ with $|A|=k$. Regardless of whether the R-game or the S-game is played, ${\rm R}^*$ is guaranteed to select a vertex of each pair from $A$ after his $k^{\rm th}$ move. Thus, the vertices chosen by ${\rm R}^*$, after his $k^{\rm th}$ move, form a resolving set of $G$. So, $o(G)=\mathcal{R}$.
\end{proof}

Despite its simplicity, Proposition~\ref{prop:pairing-in-hypergraphs} has fine applications. (We note in passing that for the Maker-Breaker domination game a parallel concept of {\em pairing dominating sets} was introduced in~\cite{gledel-2019}.) If $\dim(G)=k$ and $A$ is a pairing resolving set of $G$ with $|A| = k$, then we say that $A$ is a {\em dim-pairing resolving set} of $G$; this, together with Proposition~\ref{prop:pairing-in-hypergraphs}, immediately yields the following

\begin{corollary}\label{cor:pairing_resolving}
If $G$ admits a dim-pairing resolving set, then $R_{\rm MB}(G)=R'_{\rm MB}(G)=\dim(G)$.
\end{corollary}

For an example, consider the graph $G$ of Fig.~\ref{fig_pair_resolving}. The graph $G$ has the following dim-pairing resolving sets:
\begin{itemize}
\item $\{\{u_1, w_1\}, \{u_2, w_2\}, \{u_3, w_3\}, \{u_4, v_4\}\}$,
\item  $\{\{u_1, w_1\}, \{u_2, w_2\}, \{u_3, w_3\}, \{u_4, w_4\}\}$, and
\item $\{\{u_1, w_1\}, \{u_2, w_2\}, \{u_3, w_3\}, \{v_4, w_4\}\}$,
\end{itemize}
hence Corollary~\ref{cor:pairing_resolving} implies that $R_{\rm MB}(G) = R'_{\rm MB}(G) = 4$.

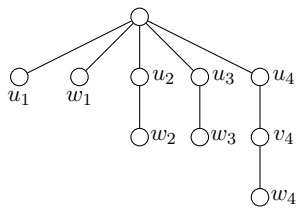
\begin{figure}[ht!]
\centering
\begin{tikzpicture}[scale=.8, transform shape]

\node [draw, shape=circle, scale=.8] (0) at  (1.5, 0) {};
\node [draw, shape=circle, scale=.8] (a) at  (-0.5, -1) {};
\node [draw, shape=circle, scale=.8] (b) at  (0.5, -1) {};
\node [draw, shape=circle, scale=.8] (2) at  (1.5, -1) {};
\node [draw, shape=circle, scale=.8] (22) at  (1.5, -2) {};
\node [draw, shape=circle, scale=.8] (3) at  (2.5, -1) {};
\node [draw, shape=circle, scale=.8] (33) at  (2.5, -2) {};
\node [draw, shape=circle, scale=.8] (4) at  (3.5, -1) {};
\node [draw, shape=circle, scale=.8] (44) at  (3.5, -2) {};
\node [draw, shape=circle, scale=.8] (444) at  (3.5, -3) {};

\node [scale=1] at (-0.5,-1.35) {$u_1$};
\node [scale=1] at (0.5,-1.35) {$w_1$};
\node [scale=1] at (1.9,-1) {$u_2$};
\node [scale=1] at (1.9,-2) {$w_2$};
\node [scale=1] at (2.9,-1) {$u_3$};
\node [scale=1] at (2.9,-2) {$w_3$};
\node [scale=1] at (3.9,-1) {$u_4$};
\node [scale=1] at (3.9,-2) {$v_4$};
\node [scale=1] at (3.9,-3) {$w_4$};

\draw(a)--(0)--(b);\draw(0)--(4)--(44)--(444);\draw(22)--(2)--(0)--(3)--(33);

\end{tikzpicture}
\caption{A graph which admits three dim-pairing resolving sets.}\label{fig_pair_resolving}
\end{figure}

To conclude the section we are going to show how pairing resolving sets can be applied to construct a large family of graphs $G$ for which $R_{\rm MB}(G) >  \dim(G)$ holds.

\begin{definition}\label{Gk}
If $k\geq 3$, then let $G_k$ be a graph of order $k+2(2^k-1)$ with $V(G_k)=A\cup B\cup C$, where $A$, $B$, and $C$ are pairwise disjoint sets with $|A|=k$ and $|B|=2^k-1=|C|$. The edge set of $G_k$ is specified as follows: (i) each of the sets $A$, $B$, and $C$ induces a clique in $G_k$; (ii) indexing the elements of $B$ (and separately of $C$) by nonempty subsets of $A$, let $b_S\in B$ be adjacent to each vertex in $S\subseteq A$; (iii) let $b_S\in B$ be adjacent to $c_S\in C$ for each nonempty subset $S$ of $A$; (iv) there are no other edges.
\end{definition}

We will at times subscript a vertex $b\in B$ (and a vertex $c\in C$) by an element of $\mathbb{Z}_2^k$, where $1$ (resp., $0$) in the $j$-th coordinate indicates that $b$ is adjacent (resp., not adjacent) to the $j$-th vertex in $A$. See Fig.~\ref{binary} for $G_3$ and the labeling of its vertices. From now on, for a given vertex $x\in V(G)$ and an ordered set of vertices $S\subseteq V(G)$, by {\em $\code_S(x)$} we denote the vector of distances between $x$ and all the vertices in $S$.

\begin{figure}[ht]
\centering
\begin{tikzpicture}[scale=.67, transform shape]

\node [draw, shape=circle, scale=.8] (a1) at  (-1, 5) {};
\node [draw, shape=circle, scale=.8] (a2) at  (0, 3) {};
\node [draw, shape=circle, scale=.8] (a3) at  (-1, 1) {};

\node [draw, shape=circle, scale=.8] (b1) at  (4, 6) {};
\node [draw, shape=circle, scale=.8] (b2) at  (4, 5) {};
\node [draw, shape=circle, scale=.8] (b3) at  (4, 4) {};
\node [draw, shape=circle, scale=.8] (b4) at  (4, 3) {};
\node [draw, shape=circle, scale=.8] (b5) at  (4, 2) {};
\node [draw, shape=circle, scale=.8] (b6) at  (4, 1) {};
\node [draw, shape=circle, scale=.8] (b7) at  (4, 0) {};

\node [draw, shape=circle, scale=.8] (c1) at  (8, 6) {};
\node [draw, shape=circle, scale=.8] (c2) at  (8, 5) {};
\node [draw, shape=circle, scale=.8] (c3) at  (8, 4) {};
\node [draw, shape=circle, scale=.8] (c4) at  (8, 3) {};
\node [draw, shape=circle, scale=.8] (c5) at  (8, 2) {};
\node [draw, shape=circle, scale=.8] (c6) at  (8, 1) {};
\node [draw, shape=circle, scale=.8] (c7) at  (8, 0) {};

\node [scale=1.3] at (-1.4,5) {$a_1$};
\node [scale=1.3] at (-0.4,3) {$a_2$};
\node [scale=1.3] at (-1.4,1) {$a_3$};

\node [scale=1.3] at (4.5, 6.3) {$b_{100}$};
\node [scale=1.3] at (4.5, 5.3) {$b_{010}$};
\node [scale=1.3] at (4.5,4.3) {$b_{001}$};
\node [scale=1.3] at (4.5, 3.3) {$b_{110}$};
\node [scale=1.3] at (4.5, 2.3) {$b_{101}$};
\node [scale=1.3] at (4.5,1.3) {$b_{011}$};
\node [scale=1.3] at (4.5,0.3) {$b_{111}$};

\node [scale=1.3] at (8.65, 6) {$c_{100}$};
\node [scale=1.3] at (8.65, 5) {$c_{010}$};
\node [scale=1.3] at (8.65,4) {$c_{001}$};
\node [scale=1.3] at (8.65, 3) {$c_{110}$};
\node [scale=1.3] at (8.65, 2) {$c_{101}$};
\node [scale=1.3] at (8.65,1) {$c_{011}$};
\node [scale=1.3] at (8.65,0) {$c_{111}$};

\node [scale=1.3] at (4.9,-1.2) {$K_7$};
\node [scale=1.3] at (9,-1.2) {$K_7$};

\draw(a1)--(a2)--(a3)--(a1);
\draw(b1)--(c1);\draw(b2)--(c2);\draw(b3)--(c3);\draw(b4)--(c4);\draw(b5)--(c5);\draw(b6)--(c6);\draw(b7)--(c7);
\draw(b1)--(a1)--(b4);\draw(b5)--(a1)--(b7);\draw(b2)--(a2);\draw(a2)--(b4);\draw(b6)--(a2)--(b7);\draw(b3)--(a3)--(b5);\draw(b6)--(a3)--(b7);

\draw[thick,dashed] (4.35,3) ellipse (1.13cm and 4cm);
\draw[thick,dashed] (8.39,3) ellipse (1.18cm and 4cm);

\end{tikzpicture}
\caption{\small $G_3$ satisfying $o(G_3) = \mathcal{R}$ and $R_{\rm MB}(G_3)>\dim(G_3)=3$.}\label{binary}
\end{figure}
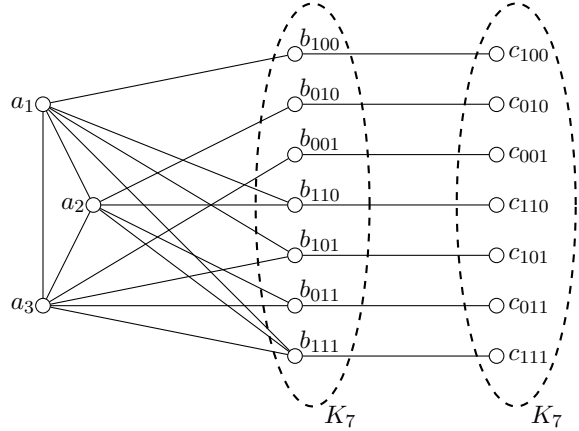

\begin{theorem}\label{Gk_claims}
If $k \ge 3$ and $G_k$ is as in Definition~\ref{Gk}, then the following holds. 
\begin{itemize}
\item[(i)] $\dim(G_k)=k$.
\item[(ii)] The set $A$ is the unique metric basis of $G_k$.
\item[(iii)] The set $\displaystyle\bigcup_{\emptyset\neq S\subseteq A}\{\{b_S,c_S\}\}$ is a pairing resolving set of $G_k$.
\end{itemize}
\end{theorem}

\begin{proof}
(i) First, we show that $A$ is a metric basis of $G_k$. Clearly, $A$ forms a resolving set of $G_k$; thus, $\dim(G_k) \le |A| =k$. To show $\dim(G_k) \ge k$, suppose $S$ is a metric basis of $G_k$ with $S \neq A$. If $S \cap A=\emptyset$, then, for any distinct $\alpha$ and $\beta$, $S \cap \{b_{\alpha}, b_{\beta}, c_{\alpha}, c_{\beta}\} \neq \emptyset$; otherwise, $\code_S(b_{\alpha})=\code_S(b_{\beta})$ and $\code_S(c_{\alpha})=\code_S(c_{\beta})$; then, $|S|\ge 2^k-2 \ge k+1$ for $k\ge 3$. Thus, $S\cap A \neq \emptyset$. By relabeling the vertices of $A$ if necessary, we can assume that $S \cap A=\cup_{i=1}^{t}\{a_i\}$, where $1 \le t \le k-1$. Then, there are $2^{k-t}$ vertices of $B$ that are not resolved by $S \cap A$, and there are $2^{k-t}$ vertices of $C$ that are not resolved by $S \cap A$; thus, $|S \cap (V(G_k)-A)| \ge 2^{k-t}-1 \ge k-t$, where the last inequality holds since $2^x-x-1 \ge 0$ for $x\ge 1$. So, $\dim(G_k)=|S|=|S\cap A|+|S\cap(V(G_k)-A)| \ge t+(k-t)=k$. Thus, $\dim(G_k)=k$. \\

(ii) Suppose $S\neq A$ is a metric basis of $G_k$; then, we have $S \cap A \neq \emptyset$ from above argument. If $|S \cap A|=t \le k-2$, then $|S \cap (V(G_k)-A)| \ge 2^{k-t}-1 \ge k-t+1$ for $k-t \ge 2$, where the last inequality holds since $2^x-x-2 \ge 0$ for $x\ge 2$. If $|S \cap A|=k-1$, then there exist ($2^{k-1}-1$) pairs in $B$ not resolved by $S\cap A$. So, $|S \cap(V(G_k)-A)| \ge 2^{k-1}-2\geq 2$ for $k\geq 3$, and thus $|S|\geq k+1$. In both cases, we find $|S|>k$, contradicting the assumption of $S$ being a metric basis. \\

(iii) First, note that $B$ and $C$ are resolving sets of $G_k$. To see that $B$ resolves $G_k$, set $B = \{b_1,\ldots, b_{2^k-1}\}$. Notice that $\code_B(c_i)$ has $1$ in the $i$th entry and $2$ in the rest of its entries, while $\code_B(a_j)$ has $1$ in exactly $2^{k-1}$ of its entries and $2$ in the rest of its entries. And $\code_B(a_i) \neq \code_B(a_j)$ for $i\neq j$, since there is $b\in B$ such that $N(b)\cap A=\{a_i\}$. The set $C$ is seen to be a resolving set by a very similar argument.

\medskip

Now, let $R=\cup_{\alpha=1}^{2^k-1}\{x_{\alpha}\}$, where $x_{\alpha} \in \{b_{\alpha}, c_{\alpha}\}$, and assume that $R \neq B$ and $R \neq C$. We show that $R$ resolves any two vertices of $G_k$ by considering memberships of the two vertices with respect to the sets $A$, $B$, and $C$; there are altogether six cases to consider.

\medskip

For distinct vertices $a_i,a_j \in A$, there exists a vertex $b_{\gamma}\in B$ such that $d(b_{\gamma}, a_i)=2=1+d(b_{\gamma}, a_j)$ and $d(c_{\gamma}, a_i)=3=1+d(c_{\gamma}, a_j)$. Since $R \cap \{b_{\gamma}, c_{\gamma}\} \neq \emptyset$, $\code_R(a_i) \neq \code_R(a_j)$.

\medskip

Let distinct vertices $b_{\alpha}, b_{\beta} \in B$ be given. If $R \cap \{b_{\alpha}, b_{\beta}\} \neq \emptyset$, then $\code_R(b_{\alpha}) \neq \code_R(b_{\beta})$. If $R \cap \{b_{\alpha}, b_{\beta}\}= \emptyset$, then $\{c_{\alpha}, c_{\beta}\} \subset R$. Then $d(c_{\alpha}, b_{\beta})=2=1+d(c_{\alpha}, b_{\alpha})$, and thus $\code_R(b_{\alpha})\neq \code_R(b_{\beta})$. The case of distinct vertices $c_{\alpha}, c_{\beta} \in C$ is handled in the same manner.

\medskip

If $a_i \in A$ and $b_{\alpha} \in B$, then $\code_R(a_i) \neq \code_R(b_{\alpha})$ since $R \cap \{b_{\alpha}, c_{\alpha}\} \neq \emptyset$.

\medskip

Let $a_i \in A$ and $c_{\beta} \in C$ be given. There exists a vertex $c_{\gamma}\in R$ such that $d(c_{\gamma}, c_{\beta})\leq 1<2\leq d(c_{\gamma}, a_i)$, and thus $\code_R(a_i) \neq \code_R(c_{\beta})$.

\medskip

Let $b_{\alpha} \in B$ and $c_{\beta} \in C$ be given. If $\alpha=\beta$, then $|R \cap \{b_{\alpha}, c_{\alpha}\}|=1$, and thus $\code_R(b_{\alpha}) \neq \code_R(b_{\beta})$. If $\alpha \neq \beta$ and $R \cap \{b_{\alpha}, c_{\beta}\} \neq \emptyset$, then $\code_R(b_{\alpha}) \neq \code_R(c_{\beta})$. If $\alpha \neq \beta$ and $R \cap \{b_{\alpha}, c_{\beta}\} = \emptyset$, then there exists $\gamma\not\in\{\alpha, \beta\}$ such that $|\{b_{\gamma},c_{\gamma}\}\cap R|=1$, and this yields $\code_R(b_{\alpha})\neq\code_R(c_{\beta})$: taking the case $c_{\gamma}\in R$ for example, we have $d(c_{\gamma}, b_{\alpha})=2=1+d(c_{\gamma}, c_{\beta})$.~\hfill
\end{proof}

From Theorem~\ref{Gk_claims}, we immediately conclude the following

\begin{corollary}
For each $k\geq 3$, we have $o(G_k)=\mathcal{R}$ and $R_{\rm MB}(G_k)>\dim(G_k)$.
\end{corollary}

\section{Some applications}
\label{sec:applications}

With the help of the results from the previous section, we now determine $o(G)$ for some classes of graphs $G$. We also determine $R_{\rm MB}(G)$ and $R'_{\rm MB}(G)$ when ${\rm R}^*$ has a winning strategy, and we determine $S_{\rm MB}(G)$ and $S'_{\rm MB}(G)$ when ${\rm S}^*$ has a winning strategy.

\subsection*{Trees}

Fix a tree $T$. A {\em support vertex} is a vertex that is adjacent to a vertex of degree one, a \emph{major vertex} is a vertex of degree at least three. A vertex $\ell$ of degree $1$ is called a \emph{terminal vertex} of a major vertex $v$ if $d(\ell, v)<d (\ell, w)$ for every other major vertex $w$ in $T$. The \emph{terminal degree}, $ter(v)$, of a major vertex $v$ is the number of terminal vertices of $v$ in $T$, and an \emph{exterior major vertex} is a major vertex that has positive terminal degree. We denote by $ex(T)$ the number of exterior major vertices of $T$, and $\sigma(T)$ the number of leaves of $T$. Let $M(T)$ be the set of exterior major vertices of $T$. Let $M_1(T)=\{w\in M(T): ter(w)=1\}$ and let $M_2(T)=\{w \in M(T):ter(w) \ge 2\}$; note that $M(T)=M_1(T) \cup M_2(T)$. For each $v \in M(T)$, let $T_v$ be the subtree of $T$ induced by $v$ and all vertices belonging to the paths joining $v$ with its terminal vertices, and let $L_v$ be the set of terminal vertices of $v$ in $T$.

\begin{theorem}\emph{\cite{tree1, tree2, tree3}}\label{dim_tree}
If $T$ is a tree that is not a path, then $\dim(T)=\sigma(T)-ex(T)$.
\end{theorem}

\begin{theorem}\emph{\cite{tree3}}\label{zhang_tree}
Let $T$ be a tree with $ex(T)=k \ge 1$, and let $v_1, \ldots, v_k$ be the exterior major vertices of $T$. For each $i \in [k]$, let $\ell_{i,1}, \ldots, \ell_{i, \sigma_i}$ be the terminal vertices of $v_i$ with $ter(v_i)=\sigma_i \ge 1$, and let $P_{i,j}$ be the $v_i-\ell_{i,j}$ path, where $j \in  [\sigma_i]$. Let $W \subseteq V(T)$. Then $W$ is a metric basis of $T$ if and only if $W$ contains exactly one vertex from each of the paths $P_{i,j}-v_i$, where $j \in [\sigma_i]$ and $i\in [k]$, with exactly one exception for each $i\in [k]$ and $W$ contains no other vertices of $T$.
\end{theorem}

\begin{theorem}
If $T$ is a tree that is not a path, then,
\begin{equation*}
o(T)=\left\{
\begin{array}{ll}
\mathcal{S}; & |N(v) \cap L_v| \ge 4 \mbox{ for some } v \in M_2(T),\\
{} &  \mbox{or } |N(u) \cap L_u|=3=|N(w) \cap L_w|  \mbox{ for distinct } u,w \in M_2(T),\\
\mathcal{N}; & |N(w) \cap L_w|=3 \mbox{ for exactly one } w\in M_2(T)\\
{} & \mbox{and } |N(v) \cap L_v| \le 2 \mbox{ for each } v \in M_2(T)-\{w\},\\
\mathcal{R}; & |N(v) \cap L_v| \le 2 \mbox{ for each } v \in M_2(T).
\end{array}\right.
\end{equation*}
Moreover, if $o(T)=\mathcal{S}$, then $S_{\rm MB}(T)=S'_{\rm MB}(T)=2$; if $o(T)=\mathcal{R}$, then $R_{\rm MB}(T)=R'_{\rm MB}(T)=\dim(T)=\sigma(T)-ex(T)$.
\end{theorem}

\begin{proof}
Let $T$ be a tree that is not a path. Hence $ex(T) \ge 1$.

First, suppose that there exists an exterior major vertex $x \in M_2(T)$ such that $|N(x) \cap L_x| \ge 4$. Since $N(x) \cap L_x$ is a twin equivalence class of cardinality at least 4, by Proposition~\ref{prop_twin}(a), $o(T)=\mathcal{S}$ and $S_{\rm MB}(T)=S'_{\rm MB}(T)=2$.

Second, suppose that $|N(v) \cap L_v| \le 3$ for each $v \in M_2(T)$. If there exist distinct $x,y \in M_2(T)$ such that $|N(x) \cap L_x|=|N(y) \cap L_y|=3$, then  $N(x) \cap L_x$ and $N(y) \cap L_y$ are distinct twin equivalence classes of cardinality 3; thus, by Proposition~\ref{prop_twin}(b), $o(T)=\mathcal{S}$ and $S_{\rm MB}(T)=S'_{\rm MB}(T)=2$.

Now, suppose there exists exactly one $z\in M_2(T)$ with $|N(z) \cap L_z|=3$, and $|N(v) \cap L_v| \le 2$ for each $v\in M_2(T)-\{z\}$. Let $L_z=\{\ell'_1, \ldots, \ell'_a\}$, where $a \ge 3$ and $d(z, \ell'_i)=1$ for $i\in [3]$; if $a \ge 4$, let $s'_j$ be the support vertex that lies on the $z-\ell'_j$ path for each $j \in [a]-[3]$. Note that, for any resolving set $W$ of $T$, Observation~\ref{obs_twin} yields $|W \cap \{\ell'_1, \ell'_2, \ell'_3\}| \ge 2$. If there exists a vertex $v\in M_2(T)-\{z\}$ with $|N(v) \cap L_v| \le 2$, then, for a fixed $w\in M_2(T)-\{z\}$, let $L_w=\{\ell_1, \ldots, \ell_{b}\}$ such that $d(w, \ell_1) \le \cdots \le d(w, \ell_b)$; if $b\ge3$, let $s_i$ be the support vertex that lies on the $w-\ell_i$ path for each $i\in [b]-[2]$. Then $R_w=\{\{\ell_1, \ell_2\}\} \cup (\cup_{i=3}^{b}\{\{s_i, \ell_i\}\})$ is a dim-pairing resolving set of $T_w$. In the S-game, ${\rm S}^*$ can occupy two vertices of $\{\ell'_1, \ell'_2, \ell'_3\}$ after her second move; thus ${\rm R}^*$ fails to occupy vertices that form a resolving set of $G$, and hence ${\rm S}^*$ wins. In the R-game, ${\rm R}^*$ can occupy two vertices of $\{\ell'_1, \ell'_2, \ell'_3\}$ after his second move, and occupy exactly one vertex of each pair in $(\cup_{i=4}^{a}\{\{s'_i, \ell'_i\}\})\cup(\cup_{w\in M_2(T)-\{z\}} R_w)$ thereafter until he completes his $\dim(T)^{\rm th}$ move; thus, the set of vertices selected by ${\rm R}^*$, after his $\dim(T)^{\rm th}$ move, forms a resolving set of $T$, and hence ${\rm R}^*$ wins. Therefore, $o(G)=\mathcal{N}$.

Third, suppose that $|N(v) \cap L_v| \le 2$ for each $v \in M_2(T)$. For a fixed $v \in M_2(T)$ with $ter(v)=k \ge 2$, let $\ell_1, \ldots, \ell_k$ be the terminal vertices of $v$ such that $d(v, \ell_1) \le \cdots \le d(v, \ell_k)$; if $d(v, \ell_i) \ge 2$, let $s_i$ be the support vertex that lies on the $v-\ell_i$ path. Then $W_v=\{\{\ell_1, \ell_2\}\} \cup (\cup_{i=3}^{k}\{\{s_i, \ell_i\}\})$ is a dim-pairing resolving set of $T_v$, and $\cup_{v\in M_2(T)} W_v$ is a dim-pairing resolving set of $T$. So, $o(T)=\mathcal{R}$ by Proposition~\ref{prop:pairing-in-hypergraphs}, and $R_{\rm MB}(T)=R'_{\rm MB}(T)=\dim(T)$ by Corollary~\ref{cor:pairing_resolving}.
\end{proof}

\subsection*{The Petersen graph}

For the Petersen graph  $\mathcal{P}$ (see Fig.~\ref{fig_petersen}) we first recall the following results.

\begin{theorem}\emph{\cite{petersen}}\label{dim_petersen}
For the Petersen graph $\mathcal{P}$, $\dim(\mathcal{P})=3$.
\end{theorem}

\begin{porism}\emph{\cite{cdim}}\label{porism}
If $W$ is a metric basis of the Petersen graph $\mathcal{P}$, then the subgraph of $\mathcal{P}$ induced by $W$ is an edge-less graph.
\end{porism}

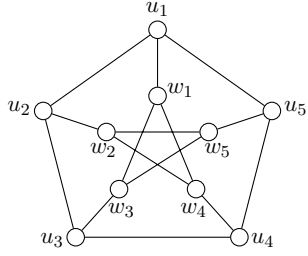
\begin{figure}[ht]
\centering
\begin{tikzpicture}[scale=.8, transform shape]

\node [draw, shape=circle, scale=.8] (1) at  (0,2) {};
\node [draw, shape=circle, scale=.8] (2) at  (-1.9, 0.65) {};
\node [draw, shape=circle, scale=.8] (3) at  (-1.36, -1.45) {};
\node [draw, shape=circle, scale=.8] (4) at  (1.36, -1.45) {};
\node [draw, shape=circle, scale=.8] (5) at  (1.9, 0.65) {};
\node [draw, shape=circle, scale=.8] (11) at  (0,0.9) {};
\node [draw, shape=circle, scale=.8] (22) at  (-0.85, 0.3) {};
\node [draw, shape=circle, scale=.8] (33) at  (-0.64, -0.65) {};
\node [draw, shape=circle, scale=.8] (44) at  (0.64, -0.65) {};
\node [draw, shape=circle, scale=.8] (55) at  (0.85, 0.3) {};

\node [scale=1] at (0,2.35) {$u_1$};
\node [scale=1] at (-2.3,0.7) {$u_2$};
\node [scale=1] at (-1.75,-1.5) {$u_3$};
\node [scale=1] at (1.75,-1.5) {$u_4$};
\node [scale=1] at (2.3,0.7) {$u_5$};
\node [scale=1] at (0.38,1) {$w_1$};
\node [scale=1] at (-0.9,0) {$w_2$};
\node [scale=1] at (-0.6,-1) {$w_3$};
\node [scale=1] at (0.6,-1) {$w_4$};
\node [scale=1] at (1,0) {$w_5$};

\draw(1)--(2)--(3)--(4)--(5)--(1);\draw(11)--(33)--(55)--(22)--(44)--(11);\draw(1)--(11);\draw(2)--(22);\draw(3)--(33); \draw(4)--(44);\draw(5)--(55);

\end{tikzpicture}
\caption{The Petersen graph.}\label{fig_petersen}
\end{figure}

\begin{lemma}\label{petersen_characterization}
Let the vertices of the Petersen graph $\mathcal{P}$ be labeled as in Fig.~\ref{fig_petersen}. Let $W_1=\{u_1,w_2,w_3\}$, $W_2=\{u_1,u_4,w_2\}$, $W_3=\{u_1,w_4,w_5\}$, $W_4=\{u_1,u_3,w_5\}$, $W_5=\{u_1,u_4,w_3\}$ and $W_6=\{u_1,u_3,w_4\}$. Then $W$ is a metric basis of $\mathcal{P}$ with $u_1\in W$ if and only if $W=W_i$ for some $i\in[6]$.
\end{lemma}

\begin{proof}

($\Leftarrow$) Let $W=W_1=\{u_1,w_2,w_3\}$. Then $\code_W(u_2)=(1,1,2)$, $\code_W(u_3)=(2,2,1)$, $\code_W(u_4)=(2,2,2)$, $\code_W(u_5)=(1,2,2)$, $\code_W(w_1)=(1,2,1)$, $\code_W(w_4)=(2,1,2)$ and $\code_W(w_5)=(2,1,1)$. So, $W_1$ is a metric basis of $\mathcal{P}$ by Theorem~\ref{dim_petersen}. For $i\in[6]-[1]$, one can easily check that $W_i$ is a metric basis of $\mathcal{P}$.

($\Rightarrow$) Let $W$ be a metric basis of $\mathcal{P}$ with $u_1\in W$. By Porism~\ref{porism}, $W \cap \{u_2,w_1,u_5\}=\emptyset$; thus, $W \cap \{w_2, w_5\} \neq \emptyset$ or $W \cap \{w_3, w_4\} \neq \emptyset$ or $W \cap \{u_3, u_4\} \neq \emptyset$.

First, let $W \cap \{w_2, w_5\} \neq \emptyset$. If $w_2\in W$, say $X_1=\{u_1,w_2\} \subset W$, then $W \cap \{u_2, u_5, w_1, w_4, w_5\}=\emptyset$ by Porism~\ref{porism} and $\code_{X_1}(u_3)=\code_{X_1}(u_4)=\code_{X_1}(w_3)=(2,2)$. Since $d(u_3,w_3)=d(u_3,u_4)$, either $w_3 \in W$ (i.e, $W=W_1$) or $u_4\in W$ (i.e., $W=W_2$). Similarly, if $w_5 \in W$, then $w_4\in W$ (i.e., $W=W_3$) or $u_3\in W$ (i.e., $W=W_4$).

Second, let $W \cap \{w_3, w_4\} \neq \emptyset$. If $w_3 \in W$, say $X_2=\{u_1, w_3\} \subset W$, then $W \cap \{u_2, u_3, u_5, w_1, w_5\}=\emptyset$ by Porism~\ref{porism} and $\code_{X_2}(u_4)=\code_{X_2}(w_2)=\code_{X_2}(w_4)=(2,2)$. Since $d(w_4,w_2)=d(w_4, u_4)$, either $w_2 \in W$ (i.e, $W=W_1$) or $u_4\in W$ (i.e., $W=W_5$). Similarly, if $w_4 \in W$, then $w_5\in W$ (i.e., $W=W_3$) or $u_3\in W$ (i.e., $W=W_6$).

Third, let $W \cap \{u_3, u_4\} \neq \emptyset$. If $u_3 \in W$, say $X_3=\{u_1, u_3\} \subset W$, then $W \cap \{u_2, u_4, u_5, w_1, w_3\}=\emptyset$ by Porism~\ref{porism} and $\code_{X_3}(w_2)=\code_{X_3}(w_4)=\code_{X_3}(w_5)=(2,2)$. Since $d(w_2,w_4)=d(w_2,w_5)$, either $w_4 \in W$ (i.e, $W=W_6$) or $w_5\in W$ (i.e., $W=W_4$). Similarly, if $u_4 \in W$, then $w_2\in W$ (i.e., $W=W_2$) or $w_3\in W$ (i.e., $W=W_5$).
\end{proof}

\begin{theorem}
$o(\mathcal{P})=\mathcal{R}$ and $R_{\rm MB}(\mathcal{P})=R'_{\rm MB}(\mathcal{P})=3$.
\end{theorem}

\begin{proof}
Let the vertices of $\mathcal{P}$ be labeled as in Fig.~\ref{fig_petersen}, and let $A_1=\{w_2,w_5\}$, $A_2=\{w_3,w_4\}$, and $A_3=\{u_3,u_4\}$. First, we consider the R-game. Since $\mathcal{P}$ is vertex-transitive (see~\cite{petersen_transitive}), we may assume that ${\rm R}^*$ occupies $u_1$ after his first move. If ${\rm S}^*$ selects a vertex in $N(u_1)$ on her first move, ${\rm R}^*$ can select a vertex of an $A_i$, $i\in[3]$, on his second move; if ${\rm S}^*$ selects a vertex of an $A_j$, $j\in[3]$, on her first move, then ${\rm R}^*$ can select the other vertex of $A_j$ on his second move. If  ${\rm R}^*$ selects a vertex of $A_1=\{w_2,w_5\}$, say $w_2$, on his second move, he can select a vertex in $\{u_4, w_3\}$ on his third move; if ${\rm R}^*$ selects a vertex of $A_2=\{w_3, w_4\}$, say $w_3$, on his second move, he can select a vertex of $\{u_4, w_2\}$ on his third move; if ${\rm R}^*$ selects a vertex of $A_3=\{u_3, u_4\}$, say $u_3$, on his second move, he can select a vertex in $\{w_4, w_5\}$ on his third move. In each case, the set of vertices occupied by ${\rm R}^*$, after his third move, forms a resolving set of $\mathcal{P}$ by Lemma~\ref{petersen_characterization}.

Second, we consider the S-game. Since $\mathcal{P}$ is edge-transitive (see~\cite{petersen_transitive}), we may assume that ${\rm S}^*$ selects  $u_5$ on her first move and ${\rm R}^*$ selects $u_1$ on his first move. If ${\rm S}^*$ selects a vertex of an $A_i$, $i\in[3]$, on her second move, then ${\rm R}^*$ can select the other vertex of $A_i$ on his second move; if ${\rm S}^*$ selects a vertex of $N(u_1)-\{u_5\}=\{u_2, w_1\}$ on her second move, ${\rm R}^*$ can select a vertex of an $A_j$, $j\in[3]$, on his second move. By applying the above argument for the R-game, it is easy to see that ${\rm R}^*$ can occupy a resolving set of $\mathcal{P}$ after his third move.

Thus, $o(\mathcal{P})=\mathcal{R}$ and $R_{\rm MB}(\mathcal{P})=R'_{\rm MB}(\mathcal{P})=3$.
\end{proof}

\subsection*{Bouquet of cycles}

Let $B_m$, $m\ge 2$, be a bouquet of $m$ cycles (i.e., the vertex sum of $m$ cycles at one common vertex),  and let $w$ be the cut-vertex of $B_m$ (see Fig.~\ref{fig_bouquet}). Let $C^1, \ldots, C^m$ be the $m$ cycles of $B_m$. For each $i \in [m]$, let $P^i=C^i-w$.

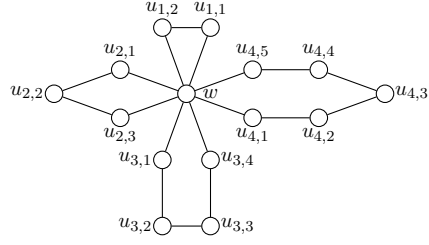
\begin{figure}[ht]
\centering
\begin{tikzpicture}[scale=.8, transform shape]

\node [draw, shape=circle, scale=.8] (0) at  (0,0) {};
\node [draw, shape=circle, scale=.8] (1) at  (-0.4,1.1) {};
\node [draw, shape=circle, scale=.8] (2) at  (0.4,1.1) {};
\node [draw, shape=circle, scale=.8] (a1) at  (-1.1,0.4) {};
\node [draw, shape=circle, scale=.8] (a2) at  (-2.2,0) {};
\node [draw, shape=circle, scale=.8] (a3) at  (-1.1,-0.4) {};
\node [draw, shape=circle, scale=.8] (b1) at  (-0.4,-1.1) {};
\node [draw, shape=circle, scale=.8] (b2) at  (-0.4,-2.2) {};
\node [draw, shape=circle, scale=.8] (b3) at  (0.4,-2.2) {};
\node [draw, shape=circle, scale=.8] (b4) at  (0.4,-1.1) {};
\node [draw, shape=circle, scale=.8] (c1) at  (1.1,-0.4) {};
\node [draw, shape=circle, scale=.8] (c2) at  (2.2,-0.4) {};
\node [draw, shape=circle, scale=.8] (c3) at  (3.3,0) {};
\node [draw, shape=circle, scale=.8] (c4) at  (2.2,0.4) {};
\node [draw, shape=circle, scale=.8] (c5) at  (1.1,0.4) {};

\node [scale=0.9] at (0.4,0) {$w$};
\node [scale=0.9] at (0.4,1.4) {$u_{1,1}$};
\node [scale=0.9] at (-0.4,1.4) {$u_{1,2}$};
\node [scale=0.9] at (-1.1,0.7) {$u_{2,1}$};
\node [scale=0.9] at (-2.65,0) {$u_{2,2}$};
\node [scale=0.9] at (-1.1,-0.7) {$u_{2,3}$};
\node [scale=0.9] at (-0.85,-1.1) {$u_{3,1}$};
\node [scale=0.9] at (-0.85,-2.2) {$u_{3,2}$};
\node [scale=0.9] at (0.85,-2.2) {$u_{3,3}$};
\node [scale=0.9] at (0.85,-1.1) {$u_{3,4}$};
\node [scale=0.9] at (1.1,-0.7) {$u_{4,1}$};
\node [scale=0.9] at (2.2,-0.7) {$u_{4,2}$};
\node [scale=0.9] at (3.75,0) {$u_{4,3}$};
\node [scale=0.9] at (2.2,0.7) {$u_{4,4}$};
\node [scale=0.9] at (1.1,0.7) {$u_{4,5}$};

\draw(0)--(1)--(2)--(0);\draw(0)--(a1)--(a2)--(a3)--(0);\draw(0)--(b1)--(b2)--(b3)--(b4)--(0);
\draw(0)--(c1)--(c2)--(c3)--(c4)--(c5)--(0);

\end{tikzpicture}
\caption{A bouquet of four cycles $B_4$.}
\label{fig_bouquet}
\end{figure}

\begin{theorem}\emph{\cite{bouquet}}\label{dim_bouquet}
If $B_m$ is a bouquet of $m\ge 2$ cycles of which $x$ cycles are even, then
\begin{equation*}
\dim(B_m)=\left\{
\begin{array}{ll}
m; & x=0,\\
m+x-1; &  x \ge 1.
\end{array}\right.
\end{equation*}
\end{theorem}

\begin{lemma}\emph{\cite{bouquet}}\label{bouquet_even}
If $W$ is a resolving set of a bouquet of cycles $B_m$, $m \ge 2$, then
\begin{itemize}
\item[\emph{(a)}] for each $i\in [m]$, $|W \cap V(P^i)| \ge 1$; and
\item[\emph{(b)}] for any two distinct even cycles $C^i$ and $C^j$ of $B_m$, $| W \cap (V(P^i) \cup V(P^j))| \ge 3$.
\end{itemize}
\end{lemma}

\begin{theorem}
If $B_m$ is a bouquet of $m\ge 2$ cycles of which $z$ are $4$-cycles,  then
\begin{equation*}
o(B_m)=\left\{
\begin{array}{ll}
\mathcal{R}; & z \le 2,\\
\mathcal{N}; & z = 3,\\
\mathcal{S}; & z \ge 4.
\end{array}\right.
\end{equation*}
Moreover, if $o(B_m)=\mathcal{R}$, then $R_{\rm MB}(B_m)=R'_{\rm MB}(B_m)=\dim(B_m)$; if $o(B_m)=\mathcal{S}$, then $S_{\rm MB}(B_m)=S'_{\rm MB}(B_m)=4$.
\end{theorem}

\begin{proof}
Let $w$ be the cut-vertex of $B_m$, where $m \ge 2$. Let $C^1, \ldots, C^z$ be cycles isomorphic to $C_4$, let $C^{z+1}, \ldots, C^x$ be  even cycles that are not isomorphic to $C_4$, and let $C^{x+1}, \ldots, C^m$ be odd cycles of $B_m$; notice $x \ge z$. If $C^i$ is an odd cycle of length $2k_i+1$, let $C^i$ be given by $w, u_{i,1}, u_{i,2}, \ldots, u_{i, k_i}, u_{i, k_i+1}, \ldots, u_{i, 2k_i}, w$; if $C^j$ is an even cycle of length $2k_j$, let $C^j$ be given by $w, u_{j,1}, u_{j,2}, \ldots, u_{j, k_j-1}, u_{j,k_j}, u_{j, k_j+1}, \ldots, u_{j, 2k_j-1}, w$ (see Fig.~\ref{fig_bouquet} for the labeling of the vertices of a $B_4$). We note that $u_{i,1}$ and $u_{i,3}$ are twins for each $i\in[z]$.

\medskip\noindent
\textbf{Case 1}: $z \le 2$.\\
If $z=0$ and $x=0$, then $\cup_{i=1}^{m} \{\{u_{i, k_i}, u_{i, k_i+1}\}\}$ is a dim-pairing resolving set of $B_m$. If $z=0$ and $x \ge 1$, then $\{\{u_{1, k_1-1},u_{1, k_1+1}\}\} \cup (\cup_{i=2}^{x} \{\{u_{i,1}, u_{i,k_i-1}\},\{u_{i, k_i+1}, u_{i, 2k_i-1}\}\}) \cup (\cup_{j=x+1}^{m}\{\{u_{j, k_j}, u_{j, k_j+1}\}\})$ is a dim-pairing resolving set of $B_m$. If $z=1$, then
$$\{\{u_{1,1}, u_{1,3}\}\} \cup (\cup_{i=2}^x \{\{u_{i,1}, u_{i,k_i-1}\},\{u_{i, k_i+1}, u_{i, 2k_i-1}\}\}) \cup (\cup_{j=x+1}^{m}\{\{u_{j, k_j}, u_{j, k_j+1}\}\})$$
is a dim-pairing resolving set of $B_m$. If $z=2$, then
\begin{eqnarray*}
\{\{u_{1,1}, u_{1,3}\}, \{u_{2,1}, u_{2,3}\}, \{u_{1,2},u_{2,2}\}\}  & \cup & (\cup_{i=3}^x \{\{u_{i,1}, u_{i,k_i-1}\},\{u_{i, k_i+1}, u_{i, 2k_i-1}\}\}) \\
& \cup & (\cup_{j=x+1}^{m}\{\{u_{j, k_j}, u_{j, k_j+1}\}\})
\end{eqnarray*}
is a dim-pairing resolving set of $B_m$. So, in each case, $o(B_m)=\mathcal{R}$ by Proposition~\ref{prop:pairing-in-hypergraphs} and $R_{\rm MB}(B_m)=R'_{\rm MB}(B_m)=\dim(B_m)$ by Corollary~\ref{cor:pairing_resolving}.

\medskip\noindent
\textbf{Case 2}: $z =3$.\\
Note that $|\cup_{i=1}^{3}V(P^i)|=9$ and, for any resolving set $W$ of $B_m$, $|W \cap (\cup_{i=1}^{3}V(P^i))| \ge 5$ by Lemma~\ref{bouquet_even}(b). In the S-game, ${\rm S}^*$ can occupy five vertices of $\cup_{i=1}^{3}V(P^i)$ after her fifth move. Thus, ${\rm R}^*$ fails to occupy vertices that form a resolving set of $B_m$. In the R-game, ${\rm R}^*$ can occupy one vertex of each pair in $(\cup_{i=1}^{3}\{\{u_{i,1}, u_{i,3}\}\}) \cup (\cup_{i=4}^x \{\{u_{i,1}, u_{i,k_i-1}\},\{u_{i, k_i+1}, u_{i, 2k_i-1}\}\}) \cup (\cup_{j=x+1}^{m}\{\{u_{j, k_j}, u_{j, k_j+1}\}\})$ and two vertices of $\cup_{i=1}^3\{u_{i,2}\}$; thus, ${\rm R}^*$ can occupy vertices that form a resolving set of $B_m$. So, $o(B_m)=\mathcal{N}$.

\medskip\noindent
\textbf{Case 3}: $z \ge 4$.\\
Note that $|\cup_{i=1}^{4}V(P^i)|=12$ and, for any resolving set $W$ of $B_m$, $|W \cap (\cup_{i=1}^{4}V(P^i))| \ge 7$ by Lemma~\ref{bouquet_even}(b). Regardless of whether ${\rm S}^*$ plays first or second, ${\rm S}^*$ can occupy 6 vertices of $\cup_{i=1}^{4}V(P^i)$ after her sixth move. So, ${\rm R}^*$ fails to occupy vertices that form a resolving set of $B_m$; thus $o(B_m)=\mathcal{S}$. In determining $S_{\rm MB}(B_m)$ and $S'_{\rm MB}(B_m)$, we note that the optimal strategy for ${\rm R}^*$ is to occupy at least a vertex in each pair of $\cup_{i=1}^{z}\{\{u_{i,1}, u_{i,3}\}\}$, and the optimal strategy for ${\rm S}^*$ is to occupy two vertices each in $V(P^i)$ and $V(P^j)$ for distinct $i,j\in[z]$. By relabeling the vertices of $\cup_{i=1}^z V(P^i)$ if necessary, we may assume that the two players occupy the vertices of $B_m$ in the order of  $V(P^1), \ldots, V(P^z)$. In the S-game, ${\rm S}^*$ can occupy two vertices of $V(P^1)$ after her second move, ${\rm R}^*$ would have occupied a vertex in $\{u_{1,1}, u_{1,3}\} \subset V(P^1)$ and a vertex in $\{u_{2,1}, u_{2,3}\} \subset V(P^2)$ after his second move, and ${\rm S}^*$ can occupy two vertices of $V(P^3)$ on her third and fourth move; thus, ${\rm S}^*$ wins after her fourth move. In the R-game, ${\rm R}^*$ occupies a vertex in $\{u_{1,1}, u_{1,3}\}$ after his first move, and ${\rm S}^*$ can occupy two vertices of $V(P^2)$ after her second move (${\rm R}^*$ would have occupied a vertex in $\{u_{2,1}, u_{2,3}\}$ on his second move). If ${\rm R}^*$ occupies a vertex in $\cup_{i=1}^{3}V(P^i)$ that has not yet been taken on his third move, ${\rm S}^*$ can occupy two vertices of $V(P^4)$ on her third and fourth move. So, in the R-game, ${\rm S}^*$ wins after her fourth move. Thus, $S_{\rm MB}(B_m)=S'_{\rm MB}(B_m)=4$.
\end{proof}

\subsection*{Complete multi-partite graphs}

The metric dimension of complete multi-partite graphs was determined in~\cite{kpartite}.

\begin{theorem}\emph{\cite{kpartite}}\label{dim_kpartite}
If $G=K_{a_1, \ldots, a_k}$, where $k \ge 2$, $n=\sum_{i=1}^{k}a_i$, and $s$ is the number of partite sets of $G$ consisting of one element, then
\begin{equation*}
\dim(G)=\left\{
\begin{array}{ll}
n-k; & s = 0,\\
n+s-k-1; & s \neq 0.
\end{array}\right.
\end{equation*}
\end{theorem}

For the MBRG we have the following description.

\begin{theorem}
If $G=K_{a_1, \ldots, a_k}$, where $k \ge 2$, and $s$ is the number of partite sets of $G$ consisting of one element, then
\begin{equation*}
o(G)=\left\{
\begin{array}{ll}
\mathcal{S}; & s \ge 4 \mbox{ or } a_i\ge 4 \mbox{ for some } i \in [k],\\
{} & \ \mbox{or } s=a_i=3  \mbox{ for some } i \in [k],\\
{} &  \mbox{or } a_i=a_j=3 \mbox{ for distinct }i,j \in[k],\\
\mathcal{N}; & s=3 \mbox{ and } a_i \le 2 \mbox{ for each } i\in[k],\\
{} & \mbox{or } s \le 2 \mbox{ and } a_i=3 \mbox{ for exactly one } i\in[k],\\
\mathcal{R}; &  \max\{s, a_i\} \le 2 \mbox{ for each } i \in [k].
\end{array}\right.
\end{equation*}
Moreover, if $o(G)=\mathcal{S}$, then $S_{\rm MB}(G)=S'_{\rm MB}(G)=2$; if $o(G)=\mathcal{R}$, then $R_{\rm MB}(G)=R'_{\rm MB}(G)=\dim(G)$.
\end{theorem}

\begin{proof}
Let $V(G)$ be partitioned into $V_1, \ldots, V_k$ such that $V_i=\{u_{i,1}, \ldots, u_{i, a_i}\}$ with $|V_i|=a_i$, where $i \in [k]$ and $k \ge 2$. We may without loss of generality assume that $a_1 \le \cdots \le a_k$.

First, suppose that $s \ge 4$ or $a_i \ge 4$ for some $i\in[k]$.  If $s \ge 4$, then $\cup_{i=1}^{s} V_i$ is a twin equivalence class of cardinality at least 4. If $a_i \ge 4$ for some $i \in [k]$, then $V_k$ is a twin equivalence class of cardinality at least 4. By Proposition~\ref{prop_twin}(a), $o(G)=\mathcal{S}$ and $S_{\rm MB}(G)=S'_{\rm MB}(G)=2$.

Second, suppose that $\max\{s, a_i\} \le 3$ for each $i\in[k]$; further, let $s=3$ or $a_i=3$ for some $i\in[k]$. If $s=a_x=3$ for some $x \in [k]$ or $a_y=a_z=3$ for distinct $y,z\in[k]$, then $G$ has distinct twin equivalence classes of cardinality three; thus, by Proposition~\ref{prop_twin}(b), $o(G)=\mathcal{S}$ and $S_{\rm MB}(G)=S'_{\rm MB}(G)=2$.

So, suppose $s=3$ or $a_i=3$ for exactly one $i\in[k]$, but not both. Let $W$ be any resolving set of $G$. By Observation~\ref{obs_twin}, we have the following: (1) if $s=3$, then $|W \cap \{u_{1,1}, u_{2,1}, u_{3,1}\}| \ge 2$; (2) if $a_i=3$ for exactly one $i\in[k]$, then $a_k=3$ and $|W \cap V_k| \ge 2$. In the S-game, ${\rm S}^*$ can occupy two vertices of $\{u_{1,1}, u_{2,1}, u_{3,1}\}$ after her second move (when $s=3$), or ${\rm S}^*$ can occupy two vertices of $V_k=\{u_{k,1}, u_{k,2}, u_{k,3}\}$ after her second move (when $a_k=3$); thus, in each case, ${\rm R}^*$ fails to occupy vertices that form a resolving set of $G$. Now, we consider the R-game. If $s=3$ (and thus $a_i \le 2$ for each $i\in[k]$), then ${\rm R}^*$ can occupy two vertices of $\{u_{1,1}, u_{2,1}, u_{3,1}\}$ after his second move, and occupy exactly one vertex of each pair in $\cup_{j=4}^{k}\{\{u_{j,1}, u_{j,2}\}\}$ thereafter until he completes his $(k-1)^{\rm th}$ move. If $a_k=3$ (and thus $s \le 2$ and $a_i \le 2$ for each $i \in [k-1]$), then ${\rm R}^*$ can occupy two vertices of $V_k=\{u_{k,1}, u_{k,2}, u_{k,3}\}$, after his second move, and occupy additional vertices (if any) thereafter as follow: (1) if $s=0$, then ${\rm R}^*$ can occupy exactly one vertex of each pair in $\cup_{i=1}^{k-1}\{\{u_{i,1}, u_{i,2}\}\}$; (2) if $s=1$, then ${\rm R}^*$ can occupy exactly one vertex of each pair in $\cup_{i=2}^{k-1}\{\{u_{i,1}, u_{i,2}\}\}$; (3) if $s=2$, then ${\rm R}^*$ can occupy exactly one vertex of each pair in $\{\{u_{1,1}, u_{2,1}\}\} \cup (\cup_{i=3}^{k-1}\{\{u_{i,1}, u_{i,2}\}\})$. So, in each case of the R-game, the vertices chosen by ${\rm R}^*$ form a resolving set of $G$; thus, ${\rm R}^*$ wins. Therefore, $o(G)=\mathcal{N}$.

Third, suppose that $\max\{s, a_i\} \le 2$ for each $i\in[k]$. If $s=0$, then $\cup_{i=1}^{k}\{\{u_{i,1}, u_{i,2}\}\}$ is a dim-pairing resolving set of $G$. If $s=1$, then $\cup_{i=2}^{k}\{\{u_{i,1}, u_{i,2}\}\}$ is a dim-pairing resolving set of $G$. If $s=2$, then $\{\{u_{1,1}, u_{2,1}\}\} \cup(\cup_{i=3}^{k}\{\{u_{i,1}, u_{i,2}\}\})$ is a dim-pairing resolving set of $G$. In each case, $o(G)=\mathcal {R}$ by Proposition~\ref{prop:pairing-in-hypergraphs} and $R_{\rm MB}(G)=R'_{\rm MB}(G)=\dim(G)$ by Corollary~\ref{cor:pairing_resolving}.
\end{proof}

\subsection*{Some grid-like graphs}

The \emph{Cartesian product} $G \cp H$ of graphs $G$ and $H$ is the graph with the vertex set $V(G) \times V(H)$ such that $(u, v)$ is adjacent to $(u', v')$ if and only if either $u = u'$ and $vv' \in E(H)$, or $v= v'$ and $uu' \in E(G)$. Products $P_s \cp P_t$  are known as {\em grid graphs}. For the rest of this section we set $V(P_s) = \{u_1, \ldots, u_s\}$ and $V(P_t) = \{v_1, \ldots, v_t\} $, see Fig.~\ref{fig_grid} for the labeling of $P_8 \cp P_4$.

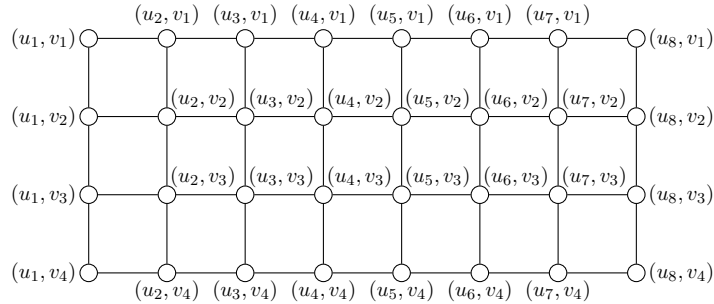
\begin{figure}[ht!]
\centering
\begin{tikzpicture}[scale=.8, transform shape]

\node [draw, shape=circle, scale=.8] (1) at  (0,0) {};
\node [draw, shape=circle, scale=.8] (2) at  (0,1.3) {};
\node [draw, shape=circle, scale=.8] (3) at  (0,2.6) {};
\node [draw, shape=circle, scale=.8] (4) at  (0,3.9) {};
\node [draw, shape=circle, scale=.8] (a1) at  (1.3,0) {};
\node [draw, shape=circle, scale=.8] (a2) at  (1.3,1.3) {};
\node [draw, shape=circle, scale=.8] (a3) at  (1.3,2.6) {};
\node [draw, shape=circle, scale=.8] (a4) at  (1.3,3.9) {};
\node [draw, shape=circle, scale=.8] (b1) at  (2.6,0) {};
\node [draw, shape=circle, scale=.8] (b2) at  (2.6,1.3) {};
\node [draw, shape=circle, scale=.8] (b3) at  (2.6,2.6) {};
\node [draw, shape=circle, scale=.8] (b4) at  (2.6,3.9) {};
\node [draw, shape=circle, scale=.8] (c1) at  (3.9,0) {};
\node [draw, shape=circle, scale=.8] (c2) at  (3.9,1.3) {};
\node [draw, shape=circle, scale=.8] (c3) at  (3.9,2.6) {};
\node [draw, shape=circle, scale=.8] (c4) at  (3.9,3.9) {};
\node [draw, shape=circle, scale=.8] (d1) at  (5.2,0) {};
\node [draw, shape=circle, scale=.8] (d2) at  (5.2,1.3) {};
\node [draw, shape=circle, scale=.8] (d3) at  (5.2,2.6) {};
\node [draw, shape=circle, scale=.8] (d4) at  (5.2,3.9) {};
\node [draw, shape=circle, scale=.8] (e1) at  (6.5,0) {};
\node [draw, shape=circle, scale=.8] (e2) at  (6.5,1.3) {};
\node [draw, shape=circle, scale=.8] (e3) at  (6.5,2.6) {};
\node [draw, shape=circle, scale=.8] (e4) at  (6.5,3.9) {};
\node [draw, shape=circle, scale=.8] (f1) at  (7.8,0) {};
\node [draw, shape=circle, scale=.8] (f2) at  (7.8,1.3) {};
\node [draw, shape=circle, scale=.8] (f3) at  (7.8,2.6) {};
\node [draw, shape=circle, scale=.8] (f4) at  (7.8,3.9) {};
\node [draw, shape=circle, scale=.8] (g1) at  (9.1,0) {};
\node [draw, shape=circle, scale=.8] (g2) at  (9.1,1.3) {};
\node [draw, shape=circle, scale=.8] (g3) at  (9.1,2.6) {};
\node [draw, shape=circle, scale=.8] (g4) at  (9.1,3.9) {};

\node [scale=0.9] at (-0.75,3.9) {$(u_1,v_1)$};
\node [scale=0.9] at (-0.75,2.6) {$(u_1,v_2)$};
\node [scale=0.9] at (-0.75,1.3) {$(u_1,v_3)$};
\node [scale=0.9] at (-0.75,0) {$(u_1,v_4)$};

\node [scale=0.9] at (1.3,4.3) {$(u_2,v_1)$};
\node [scale=0.9] at (1.9,2.9) {$(u_2,v_2)$};
\node [scale=0.9] at (1.9,1.6) {$(u_2,v_3)$};
\node [scale=0.9] at (1.3,-0.3) {$(u_2,v_4)$};

\node [scale=0.9] at (2.6,4.3) {$(u_3,v_1)$};
\node [scale=0.9] at (3.2,2.9) {$(u_3,v_2)$};
\node [scale=0.9] at (3.2,1.6) {$(u_3,v_3)$};
\node [scale=0.9] at (2.6,-0.3) {$(u_3,v_4)$};

\node [scale=0.9] at (3.9,4.3) {$(u_4,v_1)$};
\node [scale=0.9] at (4.5,2.9) {$(u_4,v_2)$};
\node [scale=0.9] at (4.5,1.6) {$(u_4,v_3)$};
\node [scale=0.9] at (3.9,-0.3) {$(u_4,v_4)$};

\node [scale=0.9] at (5.2,4.3) {$(u_5,v_1)$};
\node [scale=0.9] at (5.8,2.9) {$(u_5,v_2)$};
\node [scale=0.9] at (5.8,1.6) {$(u_5,v_3)$};
\node [scale=0.9] at (5.2,-0.3) {$(u_5,v_4)$};

\node [scale=0.9] at (6.5,4.3) {$(u_6,v_1)$};
\node [scale=0.9] at (7.1,2.9) {$(u_6,v_2)$};
\node [scale=0.9] at (7.1,1.6) {$(u_6,v_3)$};
\node [scale=0.9] at (6.5,-0.3) {$(u_6,v_4)$};

\node [scale=0.9] at (7.8,4.3) {$(u_7,v_1)$};
\node [scale=0.9] at (8.4,2.9) {$(u_7,v_2)$};
\node [scale=0.9] at (8.4,1.6) {$(u_7,v_3)$};
\node [scale=0.9] at (7.8,-0.3) {$(u_7,v_4)$};

\node [scale=0.9] at (9.85,3.9) {$(u_8,v_1)$};
\node [scale=0.9] at (9.85,2.6) {$(u_8,v_2)$};
\node [scale=0.9] at (9.85,1.3) {$(u_8,v_3)$};
\node [scale=0.9] at (9.85,0) {$(u_8,v_4)$};

\draw(1)--(2)--(3)--(4);\draw(a1)--(a2)--(a3)--(a4);\draw(b1)--(b2)--(b3)--(b4);\draw(c1)--(c2)--(c3)--(c4);\draw(d1)--(d2)--(d3)--(d4);\draw(e1)--(e2)--(e3)--(e4);\draw(f1)--(f2)--(f3)--(f4);\draw(g1)--(g2)--(g3)--(g4);
\draw(1)--(a1)--(b1)--(c1)--(d1)--(e1)--(f1)--(g1);\draw(2)--(a2)--(b2)--(c2)--(d2)--(e2)--(f2)--(g2);\draw(3)--(a3)--(b3)--(c3)--(d3)--(e3)--(f3)--(g3);\draw(4)--(a4)--(b4)--(c4)--(d4)--(e4)--(f4)--(g4);

\end{tikzpicture}
\caption{Labeling of $P_8 \cp P_4$.}\label{fig_grid}
\end{figure}

Recall the following results of grid graphs.

\begin{proposition}\emph{\cite{grid}}\label{dim_grid}
If $s,t \ge 2$, then $\dim(P_s \cp P_t)= 2$.
\end{proposition}

\begin{lemma}\emph{\cite{cdim}}
Let $s,t \ge 2$, and let $W_1 = \{(u_1,v_1),(u_1,v_t)\}$, $W_2= \{(u_1,v_1),(u_s,v_1)\}$, $W_3=\{(u_1,v_t),(u_s,v_t)\}$, and $W_4=\{(u_s,v_1),(u_s,v_t)\}$. Then W is a metric basis of $P_s \cp P_t$ if and only if $W=W_i$ for some $i \in [4]$.
\end{lemma}

For the MBRG on grid graphs we have:

\begin{proposition}
\label{prop:grids}
If $s,t \ge 2$, then $o(P_s \cp P_t)=\mathcal{R}$ and
$$R_{\rm MB}(P_s \cp P_t) = R'_{\rm MB}(P_s \cp P_t) = \dim(P_s \cp P_t)=2.$$
\end{proposition}

\begin{proof}
Since $\{\{(u_1, v_1),(u_s,v_t)\}, \{(u_s,v_1),(u_1,v_t)\}\}$ is a dim-pairing resolving set of $P_s \cp P_t$, Proposition~\ref{prop:pairing-in-hypergraphs} implies that $o(P_s \cp P_t)=\mathcal{R}$. Moreover, $R_{\rm MB}(P_s \cp P_t) = R'_{\rm MB}(P_s \cp P_t) = \dim(P_s \cp P_t)$ by Corollary~\ref{cor:pairing_resolving}.
\end{proof}

Continuing with some grid-related graphs, we next study how the MBRG behaves on the torus grid graphs, that is, the Cartesian product of cycles. We recall the following results that will be used in proving Proposition~\ref{thm_cartesian_cycles}.

\begin{theorem} \emph{\cite{grid}}\label{dim_cartesian_cycles}
If $s,t \ge 3$, then 
$$\dim(C_s\cp C_t)=\left\{\begin{array}{ll}
3; & \mbox{$s$ or $t$ is odd,}\\
4; & \mbox{$s$ and $t$ are even.}
\end{array}
\right.$$
\end{theorem}

\begin{proposition} \emph{\cite{grid}}\label{cartesian_cycles_basis}
Let $s,t \ge 3$ be integers.
\begin{itemize}
\item[(a)] If $s$ is odd, let $W=\{x,y,z\} \subseteq V(C_s\cp C_t)$ such that $x,y$ are diametral in a copy of $C_s$, and $z$ is adjacent to $x$ in a copy of $C_t$. Then $W$ is a metric basis for $C_s\cp C_t$.
\item[(b)] If $s$ and $t$ are even, let $W=\{x,y,z,w\} \subseteq V(C_s\cp C_t)$ such that $x,y$ are diametral in a copy of $C_s$, $z$ is adjacent to $x$ in a copy of $C_s$, and $w$ is adjacent to $x$ in a copy of $C_t$. Then $W$ is a metric basis for $C_s\cp C_t$.
\end{itemize}
\end{proposition}

\begin{lemma}\emph{\cite{grid}}\label{lem_antipodal}
For even integers $s,t\ge4$, let $W$ be a resolving set of $C_s\cp C_t$ with $u\in W$. If $u$ and $u'$ are diametral in $C_s \cp C_t$, then $(W-\{u\})\cup\{u'\}$ is also a resolving set of $C_s \cp C_t$.
\end{lemma}

\begin{proposition}\label{thm_cartesian_cycles}
\label{prop:torus}
If $s,t \ge 3$, then $o(C_s \cp C_t)=\mathcal{R}$ and $$R_{\rm MB}(C_s \cp C_t) = R'_{\rm MB}(C_s \cp C_t) = \dim(C_s \cp C_t).$$
\end{proposition}

\begin{proof}
Let $s,t\ge 3$ be integers and consider the following two cases. 

\medskip\noindent
{\bf Case 1}: $s$ and $t$ are even.\\
Let $W=\{x,y,z,w\}$ be a metric basis for $C_s \cp C_t$ as described in Proposition~\ref{cartesian_cycles_basis}(b); note that no two vertices in $W$ are diametral in $C_s \cp C_t$. Let $x', y', z'$, and $w'$ be diametral vertices of $x,y,z$, and $w$, respectively, in $C_s \cp C_t$. Then Lemma~\ref{lem_antipodal} implies that $\{\{x,x'\}, \{y,y'\}, \{z,z'\}, \{w,w'\}\}$ is a dim-pairing resolving set of $C_s \cp C_t$, and hence Proposition~\ref{prop:pairing-in-hypergraphs} and Corollary~\ref{cor:pairing_resolving} yield the conclusion.

\medskip\noindent
{\bf Case 2}: $s$ is odd.\\
First, we consider the $R$-game. Suppose $R^*$ selects an arbitrary vertex, say $x$, in $C_s\cp C_t$ on his first move. Note that there are two distinct vertices, say $y$ and $y'$, that are diametral to $x$ in the copy of $C_s$ containing $x$ and there are two distinct vertices, say $z$ and $z'$, that are adjacent to $x$ in the copy of $C_t$ containing $x$. By Proposition~\ref{cartesian_cycles_basis}(a), the set $\{x,y^*, z^*\}$, where $y^*\in\{y,y'\}$ and $z^*\in\{z,z'\}$, is a metric basis for $C_s \cp C_t$. Since $R^*$ can select a vertex in $\{y,y'\}$ and a vertex in $\{z,z'\}$ in his second and third move, $R^*$ wins the $R$-game.

Second, we consider the $S$-game. Suppose $S^*$ selects a vertex, say $x'$, in $C_s\cp C_t$ on her first move. Then $R^*$ can choose a neighbor $x$ of $x'$ in the copy of $C_t$ that contains $x'$. Note that there are two distinct vertices, say $y$ and $y'$, that are diametral to $x$ in the copy of $C_s$ containing $x$ and there are two distinct vertices, say $z_1$ and $z_1'$ ($z_2$ and $z'_2$, respectively), that are adjacent to $y$ ($y'$, respectively) in the copy of $C_t$ containing $y$ ($y'$, respectively). By Proposition~\ref{cartesian_cycles_basis}(a), both $\{x,y, z_1^*\}$ and $\{x,y',z_2^*\}$, where $z_1^*\in\{z_1,z'_1\}$ and $z_2^*\in\{z_2,z'_2\}$, form metric bases for $C_s \cp C_t$. So, $R^*$ can select a vertex in $\{y,y'\}$ on his second move. If $R^*$ selects $y$ on his second move, he can select a vertex in $\{z_1,z'_1\}$ in his third move; if $R^*$ selects $y'$ on his second move, he can select a vertex in $\{z_2,z'_2\}$ in his third move. In each case, $R^*$ wins the $S$-game.

Therefore, we conclude that $o(C_s \cp C_t)=\mathcal{R}$ and $R_{\rm MB}(C_s \cp C_t) = R'_{\rm MB}(C_s \cp C_t) = \dim(C_s \cp C_t)$.~\hfill
\end{proof}

We conclude this paper with some open problems.

\begin{question}
It is known that determining the metric dimension of a general graph is an NP-hard problem (see~\cite{NP}). What can we say about the computational complexity of determining the outcome of MBRG?
\end{question}

\begin{question}
For product graphs $G$, such as the Cartesian product, the lexicographic product, the corona product, and the direct product, can we determine $o(G)$, as well as $R_{\rm MB}(G)$ and $R'_{\rm MB}(G)$ (or $S_{\rm MB}(G)$ and $S'_{\rm MB}(G)$)?
\end{question}

\begin{question}
It is easy to see that $R_{\rm MB}(G)=1$ ($R'_{\rm MB}(G)=1$, respectively) if and only if $G=P_n$ for $n \ge 2$. For any positive integer $k\in[\lfloor \frac{n(G)}{2}\rfloor]-\{1\}$, can we characterize graphs $G$ satisfying $R_{\rm MB}(G)=k$ as well as $R'_{\rm MB}(G)=k$?
\end{question}

\section*{Acknowledgements}

S.K.\ acknowledges the financial support from the Slovenian Research Agency (research core funding No.\ P1-0297 and projects J1-9109, J1-1693,  N1-0095, N1-0108). I.G.Y.\ began to conduct research on this project while he was visiting the University of Ljubljana, Slovenia, supported by ``Ministerio de Educaci\'on, Cultura y Deporte'', Spain, under the ``Jos\'e Castillejo'' program for young researchers (reference number: CAS18/00030).

\end{document}